\documentclass[12pt]{amsart}                                                   
\usepackage{amsmath,amssymb,latexsym, geometry, amsthm, amsfonts, multicol, color, graphicx, tikz, array}    
\usetikzlibrary{decorations.pathreplacing}
\usepackage{mleftright}             
\usepackage{verbatim}                          
                                                                                
\newtheorem{theorem}{Theorem}[section]
\newtheorem{corollary}[theorem]{Corollary}
\newtheorem{lemma}[theorem]{Lemma}
\newtheorem{proposition}[theorem]{Proposition}

\newtheorem{conjecture}[theorem]{Conjecture}

\newtheorem{remark}[theorem]{Remark}
\newtheorem{example}[theorem]{Example}

\title{A module isomorphism between $H^*_T(G/P)\otimes H^*_T(P/B)$ and $H^*_T(G/B)$}

\author{Elizabeth Drellich}
\address{Department of Mathematics, University of North Texas, Denton, TX 76201 U.S.A.}
\email{elizabeth.drellich@unt.edu}

\author{Julianna Tymoczko}
\address{Department of Mathematics, Smith College, Northampton, MA 01063 U.S.A.}
\email{jtymoczko@smith.edu}
\thanks{Both authors were partially supported by NSF grant DMS--1248171.  The second author was partially supported by an Alfred P.~Sloan Research Fellowship.}

\begin{document}

\begin{abstract}
We give an explicit (new) morphism of modules between $H^*_T(G/P) \otimes H^*_T(P/B)$ 
and $H^*_T(G/B)$ and prove (the known result) that the two modules are isomorphic.  Our map identifies submodules of the cohomology of the flag variety that are isomorphic to each of $H^*_T(G/P)$ and $H^*_T(P/B)$.  With this identification, the map is simply the product within the ring $H^*_T(G/B)$.  We use this map in two ways.  First we describe  module bases for $H^*_T(G/B)$ that are different from traditional Schubert classes and from each other.  Second we analyze a $W$-representation on $H^*_T(G/B)$ via restriction to subgroups $W_P$.  In particular we show that the character of the Springer representation on $H^*_T(G/B)$ is a multiple of the restricted representation of $W_P$ on $H^*_T(P/B)$.
\end{abstract}

\maketitle

\section{Introduction}

{\color{black}In this paper we construct a large family of distinct bases for  the equivariant cohomology of the generalized flag variety.  To do this we give an explicit formula for the Leray-Hirsch isomorphism for the fibration of flag varieties $P/B \rightarrow G/B \rightarrow G/P$.} The Leray-Hirsch theorem says   
\[H^*(P/B) \otimes H^*(G/P) \cong H^*(G/B)\] 
but does not provide an explicit map. In fact the procedure in the Leray-Hirsch theorem sends classes through series of quotients, isomorphisms, and identifications in a spectral sequence, so explicitly writing the output is challenging. Indeed, a prori the output of the Leray-Hirsch map is a basis for the cohomology of the total space as a module over the cohomology of the base space.   We bypass these subtleties by choosing explicit bases of Schubert classes for each cohomology module as a submodule of $H^*(G/B)$ and proving the isomorphism directly. 

Moreover we solve this problem in equivariant rather than ordinary cohomology. The equivariant cohomology of a variety is an enhanced version of the ordinary cohomology ring that records  information about an underlying group action on the variety. Certain computational tools can make equivariant cohomology easier to construct than ordinary cohomology, as well as permitting us to recover ordinary cohomology.  (Knutson and Tao's work computing the structure constants of the equivariant and ordinary cohomology ring of $G(k,n)$ is one example of this principle  \cite{Knutson-Tao}.)  These computational tools can be used in many important cases, including generalized flag varieties $G/B$ and partial flag varieties $G/P$ both with the left-multiplication action of the torus $T \subseteq B \subseteq P$.

We consider a presentation of the equivariant cohomology $H^*_T(G/B)$ due to Kostant and Kumar \cite{KostantKumar} and use it to construct a module isomorphism between the tensor product $H^*_T(G/P) \otimes H^*_T(P/B)$ 
and $H^*_T(G/B)$ all treated as modules over $\mathbb{C}[\mathfrak{t}^*]$.  The map naturally descends to a module isomorphism on the ordinary cohomology.  The fact that these modules are isomorphic is not new \cite[Theorem 2.1]{Dou04}.  
But in Schubert calculus people want {\em very} explicit answers---even down to specific numbers and elementary combinatorial formulas.  The main result of this paper is a pleasingly elementary identification of classes inside $H^*_T(G/B)$ that realize Leray-Hirsch and provide a useful tool for fields like Schubert calculus (see Theorem \ref{thm: distinct bases}).  

To construct our map, we identify each factor in the tensor product  $H^*_T(G/P) \otimes H^*_T(P/B)$ 
with a  submodule of $H^*_T(G/B)$.  The bilinear module isomorphism is multiplication of classes inside the ring $H^*_T(G/B)$.   More precisely our main theorem states:

\begin{theorem}
\label{thm: main theorem}
Identify $H^*_T(G/P)$ and $H^*_T(P/B)$ isomorphically with submodules of $H^*_T(G/B)$ as described in Section \ref{section: submodules of G/B}.  Then the multiplication map
\[ p \otimes q \mapsto pq\]
induces a bilinear isomorphism of modules
\[H^*_T(G/P) \otimes H^*_T(P/B) \cong H^*_T(G/B).\]
\end{theorem}

Guillemin-Sabatini-Zara realize this isomorphism differently for a class of varieties called {\em GKM spaces}, proving a combinatorial version of the Leray-Hirsch theorem for fibrations of graphs associated to GKM spaces \cite{GKMfiberbundles}.  As an application, they show their results apply to the map $P/B \to G/B \to G/P$ in classical types \cite[Section 5]{GKMfiberbundles} and apply it to a number of specific Grassmannians and symplectic varieties \cite{Balancedfiberbundles}. Establishing their hypotheses for $P/B \to G/B \to G/P$ in classical types takes longer than our direct proof for all types.  {\color{black} Indeed in their calculations they use an interesting basis whose elements are stable under an action of the Weyl group.  In a colloquial sense Guillemin-Sabatini-Zara's basis is the ``opposite" of our flow-up classes, which are closer to the classes constructed from Morse flows or Bialynicki-Birula decompositions. } 

We give two applications of our result in Section \ref{applications}.   First our explicit module isomorphism gives rise to a large collection of module bases of $H_T^*(G/B)$ indexed by the parabolic subgroups $B\subsetneq P\subsetneq G$.  An immediate corollary of the isomorphism is that if $\mathcal{B}_{G/P}$ is any module basis for $H^*_T(G/P)$ and $\mathcal{B}_{P/B}$ is any module basis for $H^*_T(P/B)$ then the products $\{bb': b \in \mathcal{B}_{G/P}, b' \in \mathcal{B}_{P/B}\}$ form a module basis for $H^*_T(G/B)$.  
{\color{black}\begin{theorem}
Fix two distinct parabolic subgroups $P \neq Q$ with $B \subsetneq P, Q\subsetneq G$.  Let $\mathcal{B}_P$ denote the product basis of $H^*_T(G/B)$ obtained from the equivariant Schubert bases for $H^*_T(G/P)$ and $H^*_T(P/B)$, and similarly for $\mathcal{B}_Q$.   Then the bases $\mathcal{B}_P$ and $\mathcal{B}_Q$ are distinct.
\end{theorem}
Theorem \ref{thm: distinct bases} proves this result for connected $G$ and Corollary \ref{cor: disconnected G} generalizes the theorem to disconnected $G$. }One way to state the core problem of Schubert calculus is: analyze combinatorially and explicitly the cohomology ring of a generalized flag variety $G/B$ in terms of the basis of Schubert classes. Thus these parabolic bases allow us to optimize the choice of basis to make particular computations in Schubert calculus as simple as possible.

As another application we show how these bases can be used to analyze a well-known action of $W$ on $H^*_T(G/B)$ called the Springer representation.  In particular Theorem \ref{thm: kostant kumar character} says that the character of the restricted action of $W_P$ on $H^*_T(G/B)$ is the scalar multiple $|W_P| \cdot \chi$ where $\chi$ is the character of the $W_P$--representation on $H^*_T(P/B)$.   

In Section \ref{main theorem} we prove Theorem \ref{thm: main theorem} in the equivariant setting.  Our proofs use what many call GKM theory, after Goresky-Kottwitz-MacPherson's algebraic--combinatorial description of equivariant cohomology rings \cite{GKMtheory}.  The GKM presentation comes with an explicit formula for the Schubert classes that is due to Billey \cite[Theorem 4]{Billey} and Anderson-Jantzen-Soergel \cite[Remark p. 298]{A-J-S}.  These tools permit an elegant combinatorial and linear-algebraic proof of Theorem \ref{thm: main theorem}.  The result then descends to ordinary cohomology (see Corollary \ref{cor: ordinary cohomology}). 

We are grateful to J.~Matthew Douglass for showing us his work on both equivariant and ordinary cohomology isomorphism and inspiring this proof; to Alexander Yong for useful discussions; and to the anonymous referee for very helpful comments.

\section{Background}

We denote by $G$ a complex reductive linear algebraic group and fix a Borel subgroup $B$. We denote the maximal torus in $B$ by $T$ and the Weyl group associated to $G/B$ by $W$.  Let $P$ be any parabolic subgroup containing $B$.  

Let $W_P$ denote the subgroup of $W$ associated to $P$.  This is also a Weyl group, specifically the Weyl group of $P/B$.  For elements $w \in W$ the length $\ell(w)$ refers to the minimal number of simple reflections required to write $w$ as a word in the generators $\{s_i: i=1,2,\ldots,n\}$ of $W$.  Let $W^P$ denote the subset of minimal-length coset representatives of $W/W_P$.  The following fact is so essential to our work that we state it explicitly here; many texts give proofs, including Bj\"{o}rner-Brenti \cite[Lemma 2.4.3]{Bjorner-Brenti}.

\begin{proposition}
\label{prop last letter}
Every minimal-length word for each element $v \in W^P$ ends in a simple reflection $s_i \not \in W_P$.
\end{proposition}

\subsection{Restricting to fixed points}\label{section: submodules of G/B}

We use a presentation of torus-equivariant cohomology that is often referred to as {\em GKM theory}, after Goresky, Kottwitz, and MacPherson \cite{GKMpaper}, though key ideas are due to many others \cite{ AtiyahBott, Kirwan, Chang-Skjelbred} (see \cite[Section 1.7]{GKMpaper} for a fuller history).  For suitable spaces $X$ the inclusion map of fixed points $X^T \hookrightarrow X$ induces an injection on cohomology $H^*_T(X) \hookrightarrow H^*_T(X^T)$.  Straightforward algebraic conditions determine the image of the injection explicitly \cite[Lemma 2.3]{Chang-Skjelbred}, though we do not use them in this manuscript.  

Through this map $H^*_T(X) \hookrightarrow H^*_T(X^T)$ we think of equivariant classes $p \in H^*_T(X)$ as collections of polynomials in $\bigoplus_{v \in X^T} \mathbb{C}[\mathfrak{t^*}] \cong H^*_T\left(  X^T \right)$. We use functional notation to describe the elements $p \in H^*_T(X)$ meaning that for each $v \in X^T$ we have $p(v) \in \mathbb{C}[\mathfrak{t}^*]$.  

GKM theory applies to varieties like $G/B$, $G/P$, and $P/B$ that have only even-dimensional ordinary cohomology \cite[Theorem 14.1(1)]{GKMpaper}.   In fact each of $G/B$, $G/P$, and $P/B$ is a CW-complex whose cells are Schubert cells indexed by the elements of $W$, $W^P$, and $W_P$ respectively.  The fixed point sets of $G/B$, $G/P$, and $P/B$ are also naturally isomorphic to $W$, $W^P$, and $W_P$.    

As modules over $\mathbb{C}[\mathfrak{t}^*]$ the equivariant cohomology of $G/B$, $G/P$, and $P/B$ each have a basis of (equivariant) Schubert classes that are 
again indexed by the elements of $W$, $W^P$, and $W_P$ respectively.  
The restrictions $\sigma_w(u)$ of each Schubert class $\sigma_w$ to each fixed point $u$ are given explicitly by what we call {\em Billey's formula} (see Section \ref{section: Billey's formula}).  The formula is the same in all three cases $G/B$, $G/P$, and $P/B$.  Thus the map that sends the Schubert class $\sigma_w \in H^*_T(G/P)$ to the corresponding Schubert class $\sigma_w \in H^*_T(G/B)$ is a module isomorphism onto its image, and similarly for $P/B$.   We identify the images of $H^*_T(G/P)$ and $H^*_T(P/B)$ in $H^*_T(G/B)$ with the modules $H^*_T(G/P)$ and $H^*_T(P/B)$ themselves, so 
\[H^*_T(G/P) \cong \textup{span}_{\mathbb{C}[\mathfrak{t}^*]}\langle \sigma_v: v \in W^P \rangle \subseteq H^*_T(G/B)\]
and 
\[H^*_T(P/B) \cong \textup{span}_{\mathbb{C}[\mathfrak{t}^*]}\langle \sigma_w: w \in W_P \rangle \subseteq H^*_T(G/B).\]

{\bf For $G/P$ this inclusion is only a homomorphism of modules and not a homomorphism of rings.}

The map $H^*_T(G/P) \otimes H^*_T(P/B) \rightarrow H^*_T(G/B)$ that we consider is the ordinary product of classes inside $H^*_T(G/B)$.

\subsection{Billey's formula} \label{section: Billey's formula}

This section describes an explicit combinatorial formula for evaluating the polynomial $\sigma_v(u)$ in $\mathbb{C}[\mathfrak{t}^*]$.  Anderson, Jantzen, and Soergel originally discovered this formula~\cite{A-J-S}; Billey independently found it as well~\cite[Theorem 4]{Billey}.  While proven originally for $G/B$ it also holds for $G/P$ \cite[Theorem 7.1]{TymoczkoG/P} and $P/B$  \cite[Corollary 11.3.14]{Kum02}.  Fix a reduced word for $u=s_{b_1}s_{b_2} \cdots s_{b_{\ell(u)}}$ and define $\mathbf{r}(\mathbf{i},{u})=s_{b_1}s_{b_2} \cdots s_{b_{i-1}}(\alpha_{b_i}).$  Then

\begin{equation}
\label{eq:Billeys}
\sigma_v(u)= \sum \limits_{\substack{\text{reduced words} \\ v=s_{b_{j_1}}s_{b_{j_2}} \cdots s_{b_{j_{\ell(v)}}} }} \left( \prod \limits_{i=1}^{\ell(v)}  \mathbf{r}(\mathbf{j_i},{u}) \right).
\end{equation}

\begin{lemma}
\label{prop:billey} The polynomial $\sigma_v(u)$ has the following properties:
\begin{enumerate}
\item The polynomial $\sigma_v(u)$ does not depend on the choice of reduced word for $u$ \cite[Theorem 4]{Billey}.
\item The polynomial $\sigma_v(u)$ is homogeneous of degree $\ell(v)$ \cite[Corollary 5.2]{Billey}.
\item  \label{prop: non-zero} The polynomial $\sigma_v(u) \neq 0$ if and only if $v\leq u$\cite[Proposition 4.24]{KostantKumar} .
\item For any $u$ we have $\sigma_e(u)=1$.
\end{enumerate}
\end{lemma}
\begin{example}
Let $G/B$ have Weyl group $W=A_2$ and let $u=s_1s_2s_1$ and $v=s_1$.  The word $v$ is found as a subword of $s_1s_2s_1$ in the two places $\mathbf{s_1} s_2s_1$ and $s_1s_2\mathbf{s_1}$. 
$$\sigma_v(u)=\mathbf{r}(\mathbf{1},s_1s_2s_1) + \mathbf{r}(\mathbf{3},s_1s_2s_1)= \alpha_1 + s_1s_2(\alpha_1)=\alpha_1+ \alpha_2$$
\end{example}
\section{Main Theorem}
\label{main theorem}
This section proves the main theorem of the paper. First in the equivariant setting, and then for ordinary cohomology, we prove that the module map from $H^*(G/P)\otimes H^*(P/B)$ 
to $H^*(G/B)$ induced by $p \otimes q \mapsto pq$ is a bilinear isomorphism of modules. We show this first in the equivariant case by proving that it takes a module basis for $H_T^*(G/P)\otimes H_T^*(P/B)$
to a module basis for $H_T^*(G/B)$.  The non-equivariant case follows from the equivariant case. 

\begin{theorem}
\label{thm: lin ind}
The set of Schubert class products $\{\sigma_v \sigma_w: v \in W^P, w \in W_P\}$ is a linearly independent set over $\mathbb{C}[\mathfrak{t}^*]$.
\end{theorem}

In this section we will prove Theorem \ref{thm: lin ind} by arranging these products in the matrix 
$$A=(\sigma_v(v'w') \sigma_w(v'w'))_{(v,w), (v',w') \in W^P \times W_P}.$$  Our notational convention is to index rows by pairs $(v,w)\in W^P\times W_P$ and columns by pairs $(v',w')$ also in $W^P\times W_P$. \\
\\
We begin by establishing an order on $W^P \times W_P$.  The elements of both $W^P$ and $W_P$ are partially ordered by length; fix a total order on $W^P$ (respectively $W_P$) consistent with this partial order and extend this lexicographically to all of $W^P \times W_P$.   For instance all rows and columns corresponding to pairs in $(e,W_P)$ come before any pair in $(s_i,W_P)$.\\
\\
For the remainder of this section we will consider the matrix $A$ to have rows and columns ordered as above. The proof of Theorem \ref{thm: lin ind} is given in Section \ref{proof of thm:lin ind}.

\subsection{Key lemmas}
 
We begin with two lemmas. The first will prove that given the above ordering of its rows and columns, the matrix $A=(\sigma_v(v'w') \sigma_w(v'w'))$ is block upper-triangular. The second lemma will construct a matrix $M\cdot vN$ where $M$ is an invertible matrix and $vN$ is known to have linearly independent rows and columns.  The main theorem will then show how $A$ and $M \cdot vN$ are related.

\begin{lemma} \label{lemma: block-upper-triangular} 
The matrix $A=\left(\sigma_v(v'w')\sigma_w(v'w')\right)_{(v,w), (v',w') \in W^P \times W_P}$ is block upper-triangular.
\end{lemma}

\begin{proof}
Choose $v,v' \in W^P$.  Consider the entries of $A$ whose rows are indexed by pairs in $(v, W_P)$ and whose columns are indexed by pairs in $(v',W_P)$.  By construction this is a square $|W_P|\times |W_P|$ block.  Its entries are $(\sigma_v(v'w')\sigma_w(v'w'))$ where $w,w'$ range over all of $W_P$.  As established in Proposition \ref{prop last letter}, the last letter in every reduced word for $v' \in W^P$ is a simple reflection $s_i \not \in W_P$. Thus every reduced word for $v \in W^P$ inside $v'w$ is contained in the prefix $v'$.  Therefore $\sigma_v(v'w')=\sigma_v(v')$.  By Property \ref{prop: non-zero}  of Billey's formula $\sigma_v(v')$ is non-zero if and only if $v\leq v'$ in the Bruhat order.  Therefore whenever $\ell(v) \geq \ell(v')$ and $v \neq v'$ the entire block is zero. 
\end{proof}

\begin{example}{\rm
Consider the parabolic subgroup $W_P=\langle s_2 \rangle$ in the type $A_2$ Weyl group $ \langle s_1,s_2\rangle $. The minimal coset representatives are $ W^P=\{e,s_1, s_2s_1\}$.  Let $\sigma_{W_P}$ denote the collection $\{ \sigma_w: w\in W_P\}$.  Then the blocks of the matrix $A$ are
$$
\bordermatrix{~ & eW_P & s_1 W_P& s_2s_1 W_P \cr
\sigma_e\sigma_{W_P} & * & * &* \cr
\sigma_{s_1}\sigma_{W_P} &0 & * &* \cr
\sigma_{s_2s_1}\sigma_{W_P} & 0 & 0& * \cr}.
$$
}\end{example}

\begin{example}{\rm This example treats pairs $v,v' \in W^P$ with the same length.  Let $W_P$ be the parabolic subgroup $\langle s_3 \rangle \subset \langle s_1, s_2, s_3 \rangle$.  The elements of $W^P$ with length two are $s_1s_2, s_2s_1,$ and $s_3s_2$. The restriction of the matrix $A$ to the diagonal block where $v$ and $v'$ both have length two has form:
$$\bordermatrix{~ &s_1s_2W_P & s_2s_1 W_P & s_3s_2 W_P \cr
\sigma_{s_1s_2}\sigma_{W_P} & * &0&0 \cr
\sigma_{s_2s_1}\sigma_{W_P} & 0 &*&0 \cr
\sigma_{s_3s_2}\sigma_{W_P} & 0&0&* \cr}$$
}\end{example}

In the next lemma we show that the rows of the diagonal blocks of the matrix $A$ are linearly independent.  It is not immediately obvious that the matrices in this lemma are in fact the diagonal blocks; that result is part of the main theorem.

\begin{lemma}[Linear independence of diagonal blocks]\label{lemma: diagonal blocks}
Fix $v \in W^P$.  Assume that the elements of $W_P$ are ordered consistently with the partial order on length.  Let $M$ be the matrix defined by 
$$M_{wu}=\begin{cases} \sigma_{wu^{-1}}(v) &\text{if $u$ is a suffix of $w$} \\
0 & \text{otherwise}
\end{cases}
$$
where $w,u \in W_P$.  Define the matrix $N$  by $N_{u,w'}=\sigma_u(w')$ for $u,w' \in W_P$. Consider the algebra isomorphism $v: \mathbb{C}[\mathfrak{t}^*] \rightarrow \mathbb{C}[\mathfrak{t}^*]$ induced from the action $\alpha \mapsto v(\alpha)$.   Denote the image of $N$ under this action of $v$ by $vN$.\\
\\
Then the rows of the matrix $M \cdot vN$ are linearly independent over $\mathbb{C}[\mathfrak{t}^*]$.
\end{lemma}

{\bf Note} that $v$ does not permute the rows or columns of $N$.  For computational clarity, the isomorphism $\alpha \mapsto v(\alpha)$ is equivalent to the map $t_{\alpha} \mapsto t_{v(\alpha)}$ if one passes to $\mathfrak{t}$.

\begin{proof}
If $\ell(u) > \ell(w)$ then by construction $M_{wu}=0$.  If $\ell(u)=\ell(w)$ then $M_{wu}=0$  unless $w=u$.  Therefore $M$ is an upper-triangular matrix.  The entries on the diagonal have the form $M_{ww}=\sigma_e(v)=1$. Since $1$ is a unit in $\mathbb{C}[\mathfrak{t}^*]$ the matrix $M$ is invertible.\\
\\
We defined $N=(\sigma_u(w'))_{u,w'\in W_P}$ to be the matrix of Schubert classes in $H^*_T(P/B)$. The rows of $N$ are the Schubert class basis for $H_T^*(P/B)$ so the rows and columns of the matrix $N$ are linearly independent. The function $v$ acts on $N$ by sending each $\alpha$ to ${v(\alpha)}$. This operation is invertible and so preserves linear independence of the matrix rows.  Thus the new matrix $vN$  also has linearly independent rows.\\
\\
Since $M$ is invertible over $\mathbb{C}[\mathfrak{t}^*]$ and $vN$ has linearly independent rows over $\mathbb{C}[\mathfrak{t}^*]$  the rows of the matrix product $M \cdot vN$ are also linearly independent over $\mathbb{C}[\mathfrak{t}^*]$.
\end{proof}


\subsection{Proof of Theorem \ref{thm: lin ind}}
\label{proof of thm:lin ind}
We now show that each of the diagonal blocks of $A$ identified in Lemma \ref{lemma: block-upper-triangular} is a scalar multiple of the matrix $M\cdot vN$ defined by Lemma \ref{lemma: diagonal blocks}. This proves that the rows of matrix $A$ are linearly independent and thus the collection of Schubert class products $\{\sigma_v\sigma_w: v\in W^P, w\in W_P\}$ is linearly independent over $\mathbb{C}[\mathfrak{t}^*]$.
\begin{proof}
Consider the matrix $A=(\sigma_v(v'w')\sigma_w(v'w'))_{(v,w), (v',w') \in W^P \times W_P}$  with rows and columns ordered lexicographically subordinate to the length partial order on $W^P$ and $W_P$ described above.  \\
\\
Partition the matrix $A$ into blocks according to the pairs $v, v' \in W^P$.  Lemma \ref{lemma: block-upper-triangular} proved that $A$ is block-upper-triangular with this partition.  Now consider the blocks along the diagonal, namely the blocks of the form
$$(\sigma_{v}(vw')\sigma_w(vw'))_{w,w'\in W_P}= \sigma_v(v)\cdot (\sigma_w(vw'))_{w,w'\in W_P}$$
for each $v \in W^P$.   Proposition \ref{prop: non-zero} of Billey's formula guarantees that $\sigma_v(v)$ is non-zero so it suffices to consider the matrix $(\sigma_w(vw'))_{w,w'\in W_P}$.  We will show that $$(\sigma_w(vw'))_{w,w'\in W_P}=M\cdot vN$$ where $M$ and $vN$ are the matrices of Lemma \ref{lemma: diagonal blocks}.  Multiplying matrices gives 
\begin{center} 
\begin{tikzpicture}
\node at (-4,0) {$M \cdot vN =$};
\node at (0,0) {$\begin{pmatrix} && \\ & {M_{wu}} & \\ &&  \end{pmatrix} $};
\draw [decorate,decoration={brace, amplitude=5pt}] (-1.5,-.8)--(-1.5,.8);
\draw [decorate,decoration={brace, amplitude=5pt}] (-.9,1.2)--(.9,1.2);
\node at (0,1.6) {$\scriptscriptstyle\text{$u$ ranges over $W_P$}$};
\node at (-2,0) {\rotatebox[origin=c]{90}{$\scriptscriptstyle\text{$w$ ranges over $W_P$}$}};
\draw [fill] (1.9,0) circle [radius=0.03];

\node at (5,0) {$\begin{pmatrix} && \\ & v(\sigma_{u}(w')) & \\ &&  \end{pmatrix} $};
\draw [decorate,decoration={brace, amplitude=5pt}] (3.5,-.8)--(3.5,.8);
\draw [decorate,decoration={brace, amplitude=5pt}] (3.95,1.2)--(6.05,1.2);
\node at (5,1.6) {$\scriptscriptstyle\text{$w'$ ranges over $W_P$}$};
\node at (3,0) {\rotatebox[origin=c]{90}{$\scriptscriptstyle\text{$u$ ranges over $W_P$}$}};

\begin{scope}[shift={(3,-3.5)}]
\node at (-4,0) {$=$};
\node at (0,0) {$\begin{pmatrix} && \\ &\sum \limits_{u\in W_P} M_{wu} \cdot v(\sigma_{u}(w')) & \\ &&  \end{pmatrix} $};
\draw [decorate,decoration={brace, amplitude=5pt}] (-3,-.8)--(-3,.8);
\draw [decorate,decoration={brace, amplitude=5pt}] (-2.3,1.2)--(2.3,1.2);
\node at (0,1.6) {$\scriptscriptstyle\text{$w'$ ranges over $W_P$}$};
\node at (-3.5,0) {\rotatebox[origin=c]{90}{$\scriptscriptstyle\text{$w$ ranges over $W_P$}$}};
\end{scope}
\end{tikzpicture}
\end{center}
We now show that for any $w,w'\in W_P$ the polynomial $\sigma_w(vw')$ can be decomposed as the sum $\sum \limits_{u\in W_P} M_{wu} \cdot v(\sigma_{u}(w')) $ .  Consider Billey's formula for $\sigma_w(vw')$ and group terms according to which part of $w$ is a subword of $v$ and which part is a subword of $w'$.  More precisely:
$$\sigma_w(vw')=\sum \limits_{\substack{u \text{ a suffix}\\ \text{of } w}} \overbrace{\sigma_{wu^{-1}}(v)}^\text{part of $w$ found in $v$} \cdot \underbrace{v \sigma_u(w')}_\text{part of $w$ found in $w'$}$$

By construction of $M$ this is $\sum \limits_{u\in W_P} M_{wu} \cdot v(\sigma_{u}(w'))$.
Therefore the matrix $(\sigma_w(vw'))_{w,w'\in W_P}$ is equal to $M \cdot vN$ as desired. \\
\\
By Lemmas \ref{lemma: block-upper-triangular}  and \ref{lemma: diagonal blocks} the rows of the matrix $A$ are linearly independent over $\mathbb{C}[\mathfrak{t}^*]$. Thus the Schubert class products $\{\sigma_v\sigma_w: v \in W^P, w \in W_P\}$ are linearly independent over $\mathbb{C}[\mathfrak{t}^*]$. \end{proof}

Using Theorem \ref{thm: lin ind} we can prove Theorem \ref{thm: main theorem} in the equivariant setting.
\begin{theorem}
\label{thm:gb=gppb}
The map $p \otimes q \mapsto pq$ induces a bilinear isomorphism of modules $$H_T^*(G/P)\otimes H_T^*(P/B) \cong H_T^*(G/B).$$
\end{theorem}
\begin{proof}
 For any pair $(v,w)\in W^P\times W_P$ the polynomial degree of the homogeneous class $\sigma_v \sigma_w$ is  $\ell(v)+\ell(w)$ just like that of $\sigma_{vw}$. The map $W^P \times W_P \rightarrow W$ given by $(v,w) \mapsto vw$ is a bijection \cite{Bjorner-Brenti} and induces a bijection $\sigma_v \sigma_w\mapsto \sigma_{vw}$ which preserves polynomial degree.  Thus the set $\{\sigma_v\sigma_w: v \in W^P, w \in W_P\}$ contains the correct number of elements of each polynomial degree to be a basis of $H_T^*(G/B)$.\\
\\
By Theorem \ref{thm: lin ind} the set $\{\sigma_v\sigma_w: v \in W^P, w \in W_P\}$ is also linearly independent over $H_T^*(pt)$.  Thus it is a basis  for $H^*_T(G/B)$.
\end{proof}

The equivariant isomorphism induces a similar isomorphism in ordinary cohomology, essentially by Koszul duality.   We confirm this below, re-deriving the Leray-Hirsch isomorphism.  (Note that not all equivariant results immediately descend to ordinary cohomology; see for instance Theorem \ref{thm: distinct bases}.)

\begin{corollary}
\label{cor: ordinary cohomology}
The map $p \otimes q \mapsto pq$ induces a bilinear isomorphism of modules $$H^*(G/P)\otimes H^*(P/B) \cong H^*(G/B).$$
\end{corollary}

\begin{proof}
Let $M \subseteq \mathbb{C}[t^*]$ be the augmentation ideal, namely $M = \langle \alpha_1, \alpha_2, \ldots, \alpha_n\rangle$.  Recall that the ordinary cohomology is the quotient $H^*(X) \cong \frac{H^*_T(X)}{MH^*_T(X)}$ when $X$ satisfies certain conditions, for instance, if $X$ has no odd-dimensional ordinary cohomology \cite{GKMpaper}.  Consider the two projections 
\[H^*_T(G/P) \otimes_{\mathbb{C}[\mathfrak{t}^*]} H^*_T(P/B) \rightarrow \frac{H^*_T(G/P) \otimes_{\mathbb{C}[\mathfrak{t}^*]} H^*_T(P/B)}{M \left( H^*_T(G/P) \otimes_{\mathbb{C}[\mathfrak{t}^*]} H^*_T(P/B) \right)}\]
and 
\[H^*_T(G/P) \otimes_{\mathbb{C}[\mathfrak{t}^*]} H^*_T(P/B) \rightarrow \frac{H^*_T(G/P)}{M H^*_T(G/P)} \otimes_{\mathbb{C}[\mathfrak{t}^*]} \frac{H^*_T(P/B)}{M H^*_T(P/B)}.\]
If $a \otimes b \in H^*_T(G/P) \otimes_{\mathbb{C}[\mathfrak{t}^*]} H^*_T(P/B)$ and $m \in M$ then $m(a \otimes b) = (ma) \otimes b = a \otimes (mb)$ so the kernels of the two projections agree.  It follows that we have an isomorphism
\[\frac{H^*_T(G/P) \otimes_{\mathbb{C}[\mathfrak{t}^*]} H^*_T(P/B)}{M \left( H^*_T(G/P) \otimes_{\mathbb{C}[\mathfrak{t}^*]} H^*_T(P/B) \right)} \cong \frac{H^*_T(G/P)}{M H^*_T(G/P)} \otimes_{\mathbb{C}[\mathfrak{t}^*]} \frac{H^*_T(P/B}{M H^*_T(P/B)}.\]
The map $\phi: H^*_T(G/P) \otimes_{\mathbb{C}[\mathfrak{t}^*]} H^*_T(P/B) \rightarrow H^*_T(G/B)$ is an isomorphism of $\mathbb{C}[\mathfrak{t}^*]$-modules so it commutes with taking the quotient by the augmentation $MH^*_T(G/B)$.  Combining these results gives
\[\frac{H^*_T(G/P)}{M H^*_T(G/P)} \otimes_{\mathbb{C}[\mathfrak{t}^*]} \frac{H^*_T(P/B)}{M H^*_T(P/B)} \cong  \frac{H^*_T(G/B}{M H^*_T(G/B)}\]
or in other words $H^*(G/P) \otimes H^*(P/B) \cong H^*(G/B)$.
\end{proof}

\section{Applications}
\label{applications}

The parabolic basis $\mathcal{B}_P=\{\sigma_v\sigma_w : v\in W^P, w\in W_P\}$ is generally not the Schubert basis. In fact we will show that, with the exception of $P=G$ and $P=B$, each of the bases $\mathcal{B}_P$ is distinct not only from the Schubert basis but from any other parabolic basis as well.  As with different bases of symmetric functions, this is a useful computational tool. As another application we compute the character of a particular Springer representation.
\subsection{The parabolic basis $\mathcal{B}_P$}
We begin with an example illustrating that the basis $\mathcal{B}_P$ is not the Schubert basis.
\def\GKMgraph{\draw(A)--(B)--(C)--(D)--(E)--(F)-- cycle; 
\draw (A)--(D); \draw(B)--(E);
 \draw (F)--(C);}
\newcommand{\Schubertclass}[7]{%
    \coordinate (A) at (0,0);
    \coordinate (B) at (-1,.8);
    \coordinate (C) at (-1,2.2);
    \coordinate (D) at (0,3);
    \coordinate (E) at (1,2.2);
    \coordinate (F) at (1,.8);
    \coordinate (G) at (0,4);
    \GKMgraph
      \node [below] at (A) {${#1}$};
      \node [left] at (B) {${#2}$};
      \node [left] at (C) {${#3}$};
      \node [above] at (D) {${#4}$};
      \node [right] at (E) {${#5}$};
      \node [right] at (F) {${#6}$};
      \node at (G) {${#7}$};
   }

\begin{example}
\label{ex: a2}
We again use the $A_2$ example $W_P=\langle s_2 \rangle$ and $ W^P=\{e,s_1, s_2s_1\}$.  Four of the classes in $\mathcal{B}_P$ are also Schubert classes:
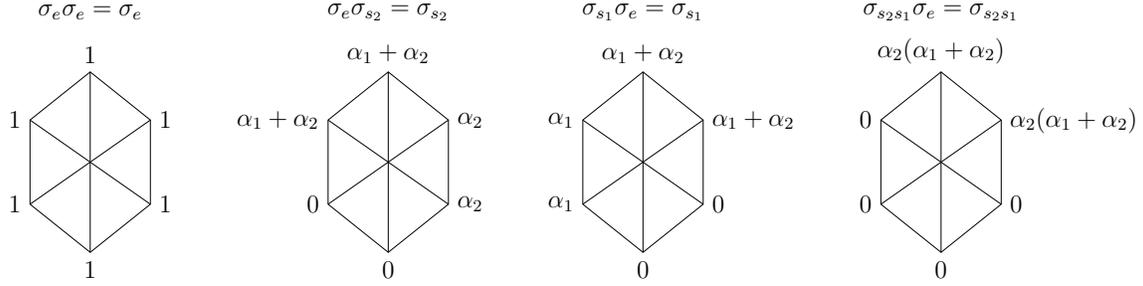
\begin{center}
\scalebox{.8}{
\begin{tikzpicture}
\matrix[column sep=0.8cm,row sep=0.5cm, ampersand replacement=\&]
{
\Schubertclass{1}{1}{1}{1}{1}{1}{\sigma_e\sigma_e=\sigma_e} \&
\Schubertclass{0}{0}{\alpha_1+\alpha_2}{\alpha_1+\alpha_2}{\alpha_2}{\alpha_2}{\sigma_e\sigma_{s_2}=\sigma_{s_2}} \&
\Schubertclass{0}{\alpha_1}{\alpha_1}{\alpha_1+\alpha_2}{\alpha_1+\alpha_2}{0}{\sigma_{s_1}\sigma_e=\sigma_{s_1}} \&
\Schubertclass{0}{0}{0}{\alpha_2(\alpha_1+\alpha_2)}{\alpha_2(\alpha_1+\alpha_2)}{0}{\sigma_{s_2s_1}\sigma_e=\sigma_{s_2s_1}}\\
};
\end{tikzpicture}
}\end{center}

The remaining two classes are not Schubert classes.

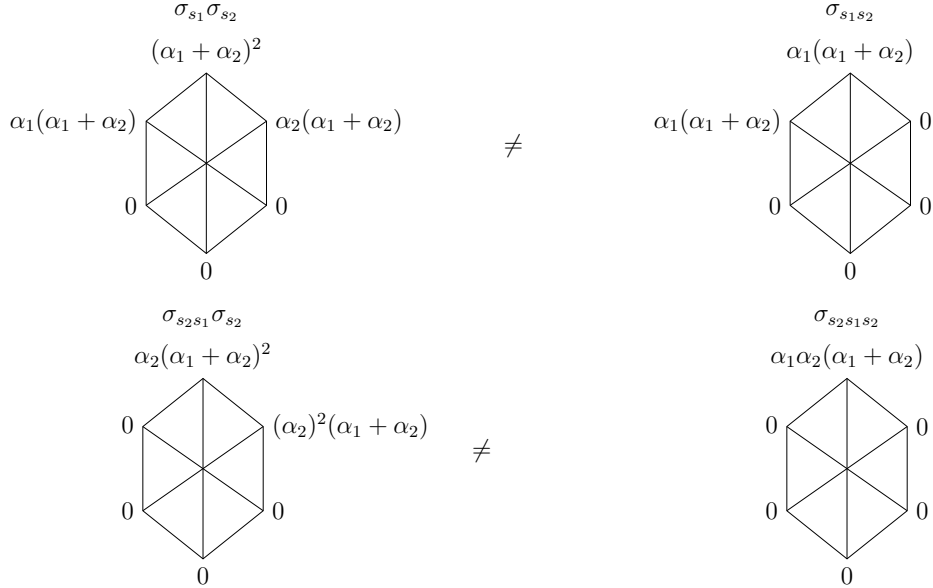
\scalebox{.8}{
\begin{tabular}{m{8cm} m{2cm} m{8cm}}
\begin{tikzpicture}
\Schubertclass{0}{0}{\alpha_1(\alpha_1+\alpha_2)}{(\alpha_1+\alpha_2)^2}{\alpha_2(\alpha_1+\alpha_2)}{0}{\sigma_{s_1}\sigma_{s_2}} 
\end{tikzpicture}
&
 $\neq$ &
\begin{tikzpicture}

 \Schubertclass{0}{0}{\alpha_1(\alpha_1+\alpha_2)}{\alpha_1(\alpha_1+\alpha_2)}{0}{0}{\sigma_{s_1s_2}};
\end{tikzpicture}
\end{tabular}}\\
\\
\scalebox{.8}{
\begin{tabular}{m{8cm} m{2cm} m{8cm}}
\begin{tikzpicture}
\Schubertclass{0}{0}{ {\color{white}\alpha_1\alpha(\alpha_1+\alpha_2)} 0}{\alpha_2(\alpha_1+\alpha_2)^2}{(\alpha_2)^2(\alpha_1+\alpha_2)}{0}{\sigma_{s_2s_1}\sigma_{s_2}} 
\end{tikzpicture}
&
 $\neq$ &
\begin{tikzpicture}
 \Schubertclass{0}{0}{{\color{white}\alpha_1\alpha(\alpha_1+\alpha_2)}0 }{\alpha_1\alpha_2(\alpha_1+\alpha_2)}{0}{0}{\sigma_{s_2s_1s_2}};
\end{tikzpicture}
\end{tabular}}\\
The class $\sigma_{s_1}\sigma_{s_2}$ is equal to $\sigma_{s_1s_2}+\sigma_{s_2s_1}$ and the class $\sigma_{s_2s_1}\sigma_{s_1}$ is equal to $\sigma_{s_1s_2s_1}+\alpha_2 \sigma_{s_2s_1}$.
\end{example}

A Schubert class $\sigma_w$ could appear in one of these parabolic bases even if the word $w$ is contained in neither the parabolic subgroup nor the set of minimal coset representatives.  If the class $\sigma_w$ does appear in the basis, we know exactly which Schubert classes are multiplied together to obtain it.

\begin{lemma}
\label{lemma: one class to two}  Fix a parabolic $P$ and suppose that $v\in W^P, w\in W_P$. 
If the product class $\sigma_v\sigma_w$ is equal to a single Schubert class $\sigma_u$ then $u=vw$ is the parabolic decomposition of $u$.
\end{lemma}
This lemma is a consequence of a result of Reiner, Woo, and Yong \cite[Lemma 2.2]{ReinerWooYong}. We give a different proof.
\begin{proof}
If $\sigma_v\sigma_w=\sigma_u$ then $\ell(v)+\ell(w)=\ell(u)$ since both sides must have the same polynomial degree. For any $u'\not \geq u$ at least one of $\sigma_v(u'), \sigma_w(u')$ must be zero since the product $\sigma_v(u')\sigma_w(u')$ is zero. By construction $\ell(v)+\ell(w)=\ell(vw)$ and by Property \ref{prop: non-zero} of Billey's formula $\sigma_v(vw)$ and $\sigma_w(vw)$ are both nonzero. Thus $\sigma_u(vw)$ is nonzero, implying that $vw \geq u$. But $vw$ has the same length as $u$ so the two words must be equal.
\end{proof}

Both $\mathcal{B}_G$ and $\mathcal{B}_B$ are the classical Schubert basis.  Since $W_G=W^B=W$ every class in $\mathcal{B}_G$ has the form $\sigma_e\sigma_w=\sigma_w$ and every class in $\mathcal{B}_B$ has the form $\sigma_v\sigma_e=\sigma_v$.  However all other parabolic bases are distinct.  For the sake of clarity we prove the result first in the case when $G$ is connected and then for general reductive linear algebraic groups $G$.

\begin{theorem}
\label{thm: distinct bases}
Assume that the Dynkin diagram for $W$ is connected.  With the exception of $P=G$ and $Q=B$, distinct parabolics $P$ and $Q$ have distinct bases $\mathcal{B}_P$ and $\mathcal{B}_Q$ for $H_T^*(G/B)$.
\end{theorem}
\begin{proof}

Neither $P$ nor $Q$ is $B$ so there is at least one simple reflection in each of $W_P$ and $W_Q$.  We also assume that $P \neq Q$ so at least one of $W_P$ and $W_Q$ contains a simple reflection that the other does not.  Without loss of generality assume that there is at least one simple reflection in $W_P\setminus W_Q$.  Consider all paths in the Dynkin diagram between simple roots corresponding to reflections in $W_Q$ and simple roots corresponding to reflections in $W_P \setminus W_Q$.  Choose a minimal-length such path and denote the endpoints by $\alpha_i$ and $\alpha_j$ where $\alpha_i$ corresponds to $s_i \in W_P\setminus W_Q$ and $\alpha_j$ corresponds to $s_j\in W_Q$.   Let $v_R$ be the word $s_j v s_i$ corresponding to that path. The path is minimal so the word $v$ contains no reflections in $W_P$ or $W_Q$.

Since $v_R\in W$ is a word whose letters $s_{i_1} s_{i_2} ... s_{i_k}$ are the simple reflections in order corresponding to the vertices of a path in the Dynkin diagram, the letters $s_{i_j}$ and $s_{i_{j+1}}$ do not commute.  Each reflection occurs at most once so no braid moves can be performed on $v_R$.  Thus $v_R$ has no other factorization into simple reflections in $W$. 

Consider the factorization of $v_R$ in each of $W^PW_P$ and $W^QW_Q$.  Since $v_R$ has a unique minimal word this factorization simply splits $v_R$ into a prefix and a suffix, with the prefix ending in the rightmost occurrence of $s_k \not \in W_P$  respectively $W_Q$.  In particular the reflection $s_i$ is not in $W_Q$ so $v_R \in W^Q$.  If $v \neq e$ then similarly $s_jv \in W^P$ and $s_i \in W_P$.

We will show that either the Schubert class corresponding to $v_R$ is in $\mathcal{B}_Q$ or the Schubert class corresponding to $s_is_j$ is in $\mathcal{B}_P$.  The two bases $\mathcal{B}_P$ and $\mathcal{B}_Q$ are equal only if $\sigma_{v_R}$ appears in $\mathcal{B}_P$ or $\sigma_{s_is_j}$ appears in $\mathcal{B}_Q$ respectively.  Lemma \ref{lemma: one class to two} showed that the only way that $\sigma_a\sigma_b \in \mathcal{B}_P$ could equal $\sigma_{v_R}$ is if $ab=v_R$ and similarly for $\sigma_{s_is_j}$.  We will then evaluate at particular Weyl group elements to prove that the  classes cannot be equal.  

The cases we consider are:
\begin{enumerate}
\item[Case 1.] If $s_j\not \in W_P$ then $\sigma_{v_R}$ is in $\mathcal{B}_Q$ and $\sigma_{s_jv} \sigma_{s_i} \in \mathcal{B}_P$.
\item[Case 2.] If $s_j\in W_P$ then
\begin{enumerate}
\item[a] if $v_R=s_jvs_i$ for $v\neq e$ then $\sigma_{v_R}$ is in $\mathcal{B}_Q$ and $\sigma_{s_jv} \sigma_{s_i} \in \mathcal{B}_P$.   
\item[b] if $v_R=s_js_i$ then $\sigma_{s_is_j}$ is in $\mathcal{B}_P$ and $\sigma_{s_i}\sigma_{s_j} \in \mathcal{B}_Q$.
\end{enumerate}
\end{enumerate}

{\bf Case 1.}  To see that $\sigma_{v_R} \neq \sigma_{s_jv}\sigma_{s_i}$ we compare their values at $s_jvs_is_jv$.   We could prove that $s_jvs_is_jv$ is reduced by an argument involving relations like the one with which we proved $v_R$ has a unique reduced word; alternatively we could observe that $s_jvs_is_jv$ is reduced because it is in Bj\"{o}rner-Brenti's {\em normal form} \cite[Proposition 3.4.2]{Bjorner-Brenti} (with roots ordered in the same order as the path from $\alpha_i$ to $\alpha_j$). On the one hand 
\[\sigma_{v_R}(s_jvs_is_jv) = \sigma_{s_jv}(s_jv) s_jv (\alpha_i) = \sigma_{v_R}(v_R).\]  
On the other hand 
\[\sigma_{s_jv}(s_jvs_is_jv)\cdot \sigma_{s_i}(s_jvs_is_jv) = \left(\sigma_{s_jv}(s_jv)+s_jvs_i\sigma_{s_jv}(s_jv)+ \text{other non-negative terms} \right) \cdot s_jv (\alpha_i)\]
which is 
$\sigma_{v_R}(v_R) + \textup{something positive}$.  This proves the claim in this case.

{\bf Case 2a.}  In this case we evaluate the classes at $s_is_jvs_i$ which is reduced by the previous argument.  In $\mathcal{B}_Q$ we have $\sigma_{v_R}(s_is_jvs_i)= s_i(\sigma_{v_R}(v_R))$ while in  $\mathcal{B}_P$ we have
$$
\sigma_{s_jv}(s_is_jvs_i)\cdot \sigma_{s_i}(s_is_jvs_i)=s_i(\sigma_{s_jv}(v_R))\cdot (\alpha_i+s_is_jv(\alpha_i)).
$$
Again this equals $s_i(\sigma_{v_R}(v_R))+\text{ something positive}$.  The claim holds in this case, too.

{\bf Case 2b.} We look at the classes corresponding to $s_is_j$.  The word $s_is_j$ is contained in $W_P$ and decomposes into $s_i\in W^Q$ and $s_j\in W_Q$.  Evaluating at the reduced word $s_js_is_j$ gives  
$$\sigma_{s_i}(s_js_is_j)\cdot \sigma_{s_j}(s_js_is_j)=s_j(\alpha_i)\cdot (\alpha_j+s_js_i(\alpha_j))$$
$$\textup{but       } \hspace{1em} \sigma_{s_is_j}(s_js_is_j)=s_j(\alpha_i)\cdot s_js_i(\alpha_j).$$
These are unequal which proves the theorem.
\end{proof}

This result easily extends to the case of the general complex reductive linear algebraic group.

\begin{corollary}
\label{cor: disconnected G}
Suppose that $G$ is not connected and denote the corresponding factorization of $W$ by $W = W_1 \times W_2 \times \cdots \times W_k$.  Suppose that $P$ and $Q$ are parabolics.  Then the bases $\mathcal{B}_P$ and $\mathcal{B}_Q$ are different bases of $H^*_T(G/B)$ if and only if  for some factor $W_i$ 
\begin{itemize}
\item $W_P \cap W_i$ differs from $W_Q \cap W_i$ and
\item at least one of $W_P\cap W_i$ and $W_Q\cap W_i$ is a nonempty proper subset of $W_i$.
\end{itemize}
\end{corollary}

\begin{proof}
If $W$ can be factored as $W = W_1 \times W_2 \times \cdots \times W_k$ then the cohomology classes in $H^*_T(G/B)$ are  products $p_1 \times p_2 \times \cdots \times p_k$ where $p_i$ is in the equivariant cohomology of the flag variety $G_i/B_i$ corresponding to $W_i$ for each $i \in \{1,2,\ldots, k\}$.  Two classes $p_1 \times p_2 \times \cdots \times p_k$ and $q_1 \times q_2 \times \cdots \times q_k$ are equal if and only if $p_i = q_i$ for each $i$.  As long as Theorem \ref{thm: distinct bases} holds for one factor $i$ the bases $\mathcal{B}_P$ and $\mathcal{B}_Q$ are distinct.  This is precisely the content of the two conditions.

Conversely suppose that for each $W_i$ either $W_P\cap W_i= W_Q \cap W_i$ or both are in the set $\{W_i,\{e\}\}$.  Let $\mathcal{B}_{P_i}$ be the basis of $H_T^*(G_i/B_i)$ corresponding to the parabolic with Weyl group $W_P\cap W_i$. Then $\mathcal{B}_P= \mathcal{B}_{P_1} \times \mathcal{B}_{P_2} \times \cdots \times \mathcal{B}_{P_k}$ and $\mathcal{B}_Q= \mathcal{B}_{Q_1} \times \mathcal{B}_{Q_2} \times \cdots \times \mathcal{B}_{Q_k}$.  If $W_P\cap W_i= W_Q \cap W_i$ then $\mathcal{B}_{P_i}=\mathcal{B}_{Q_i}$.  If both are in $\{W_i, \{e\}\}$ then both $\mathcal{B}_{P_i}$ and $\mathcal{B}_{Q_i}$ are the Schubert basis for $H_{T_i}^*(G_i/B_i)$.  In all cases $\mathcal{B}_P=\mathcal{B}_Q$ as desired.
\end{proof}

\begin{remark} The proof of Theorem \ref{thm: distinct bases} does not immediately extend to ordinary cohomology. Consider Example \ref{ex: a2}, which described $W_P= \langle s_2 \rangle$ in type $A_2$.  The basis in that case is 
\[\mathcal{B}_P=\{\sigma_e, \sigma_{s_1},\sigma_{s_2}, \sigma_{s_2s_1}, \sigma_{s_1s_2}+\sigma_{s_2s_1}, \sigma_{s_1s_2s_1}+\alpha_2\sigma_{s_2s_1}\}.\] 
If $W_Q= \langle s_2 \rangle$ then the basis $\mathcal{B}_Q$ is the set of classes 
\[\mathcal{B}_Q = \{\sigma_e, \sigma_{s_1},\sigma_{s_2}, \sigma_{s_1s_2}, \sigma_{s_1s_2}+\sigma_{s_2s_1}, \sigma_{s_1s_2s_1}+\alpha_1\sigma_{s_1s_2}\}.\] 
The highest degree element in each basis differs in equivariant cohomology but both project to  $\sigma_{s_1s_2s_1}$ in ordinary cohomology using the map in Corollary \ref{cor: ordinary cohomology}.  In particular the localizations of two basis elements could differ in equivariant cohomology but agree in ordinary cohomology, so the previous proof does not immediately apply to ordinary cohomology.

Nonetheless one basis contains $\sigma_{s_1s_2}$ while  the other contains $\sigma_{s_2s_1}$.  In other words $\mathcal{B}_P$ and $\mathcal{B}_Q$ are distinct bases for $H^*(G/B)$ even though the previous proof does not hold.  This leads us to the following conjecture. 
\end{remark}
\begin{conjecture}
The bases $\mathcal{B}_P$ and $\mathcal{B}_Q$ are distinct for $P,Q$ distinct parabolics not equal to $B$ in the ordinary cohomology.  In other words Theorem \ref{thm: distinct bases} holds for $H^*(G/B)$ as well as $H_T^*(G/B)$.
\end{conjecture}
\subsection{Representations of $W_P$}

In this section we use the parabolic basis of $H^*_T(G/B)$ to describe explicitly the character of an action of $W_P$ on $H^*_T(G/B)$.  This group action is in fact the restriction of a well-known Weyl group action on $H^*_T(G/B)$ called {\em Springer's representation}; Kostant and Kumar first studied the presentation we use here \cite[see Section 4.17 and Proposition 4.24.g]{KostantKumar}.

Kostant and Kumar showed that the Weyl group acts as a collection of algebra homomorphisms on the equivariant cohomology $H^*_T(G/B)$ according to the rule that if $w \in W$ and $p \in H^*_T(G/B)$ then the class $w \cdot p$ is given, in terms of its localizations, by:
\[(w \cdot p)(v) = p(vw^{-1}) \hspace{0.5in} \textup{ for each } v \in W.\]
We restrict this action to an arbitrary parabolic subgroup $W_P$ of $W$.

\begin{theorem}
\label{thm: kostant kumar character}
Fix a subset $P$ of simple roots in $\Delta$.  Let $m_P = |W_P|$.  Let $\chi_P$ denote the character of the restriction to $W_P$ of Kostant-Kumar's action of $W$ on $H^*_T(G/B)$.  Then 
\[\chi_P = m_P \chi \]
where $\chi$ is the character of Kostant-Kumar's action of $W_P$ on $H^*_T(P/B)$.
\end{theorem}

\begin{proof}
Consider the basis $\{\sigma_w\sigma_v: w \in W^P, v \in W_P\}$ from the previous section.  Choose any simple reflection $s_i \in W_P$ and consider the image $s_i \cdot (\sigma_{w'} \sigma_{v'})$ of $s_i$ acting on the basis element $\sigma_{w'} \sigma_{v'}$.  Kostant-Kumar's action is a map of algebras so
\[ s_i \cdot (\sigma_{w'} \sigma_{v'}) = (s_i \cdot \sigma_{w'}) (s_i \cdot \sigma_{v'}).\]
No reduced word for $w\in W^P$ ends in $s_i$.  For each $u \in W$ we conclude by Billey's formula that
\[\sigma_{w'}(us_i)=\sigma_{w'}(u).\]
Hence $s_i \cdot \sigma_{w'} = \sigma_{w'}$ and so for all $s_i, v' \in W_P$ and all $w' \in W^P$ we have
\[s_i \cdot (\sigma_{w'} \sigma_{v'})  = \sigma_{w'}  (s_i \cdot \sigma_{v'}).\]
It follows that for each $v'' \in W_P$ we have $v'' \cdot (\sigma_{w'} \sigma_{v'})  = \sigma_{w'}  (v'' \cdot \sigma_{v'})$.  In particular  $\chi_P = |W_P| \chi$ since the coefficient of $\sigma_{w'} \sigma_{v'}$ in the expansion $v'' \cdot  (\sigma_{w'} \sigma_{v'})$ in terms of the basis $\{\sigma_{w} \sigma_{v}\}$ is the same as the coefficient of $\sigma_{v'}$ in the expansion of $v'' \cdot \sigma_{v'}$ in terms of the basis $\{\sigma_v\}$.
\end{proof}

\bibliography{mybib}{}
\bibliographystyle{plain}

\end{document}